\documentclass[12pt,reqno]{amsart}
\usepackage{amssymb}

\usepackage[curve,matrix,arrow]{xy}

\baselineskip

\theoremstyle{plain} 
\numberwithin{equation}{section}
\newtheorem{thm}[equation]{Theorem}
\newtheorem{lemma}[equation]{Lemma}
\newtheorem{cor}[equation]{Corollary}
\newtheorem{prop}[equation]{Proposition}

\newtheorem{defi}[equation]{Definition}

\theoremstyle{definition}
\newtheorem{rem}[equation]{Remark}

\theoremstyle{remark}

\def\CD{{\mathcal{D}}}
\def\CE{{\mathcal{E}}}

\def\CI{{\mathcal{I}}}
\def\CJ{{\mathcal{J}}}
\def\CK{{\mathcal{K}}}

\def\FJ{{\mathfrak{J}}}

\def\bZ{{\mathbb Z}}
\def\tI{\tilde{I}}
\def\tJ{\tilde{J}}

\def\Id{\operatorname{Id}\nolimits}
\def\Dim{\operatorname{Dim}\nolimits}

\def\Ann{\operatorname{Ann}\nolimits}

\def\stmod{\operatorname{{\bf stmod}}\nolimits}

\def\HHH{\operatorname{H}\nolimits}
\def\wHH{{\widehat{\HHH}}}
\def\Hom{\operatorname{Hom}\nolimits}
\def\PHom{\operatorname{PHom}\nolimits}

\def\End{\operatorname{End}\nolimits}
\def\HH#1#2#3{\HHH^{#1}(#2,#3)}

\def\Homul{\operatorname{\underline{Hom}}\nolimits}

\def\Ext{\operatorname{Ext}\nolimits}
\def\wExt{\widehat{\Ext}}

\def\hgs{\HH{*}{G}{k}}

\def\cemp{\CE_{+}}
\def\cemm{\widehat{\CE}_{-}}
\def\hcem{\widehat{\CE}}
\def\hhg{\HHH_G}

\def\ann{\operatorname{Ann}\nolimits}

\title[Nilpotence in Complete Cohomology]
{Nilpotence and Duality in the Complete Cohomology of a Module}
\author[Jon F. Carlson]{Jon F. Carlson}
\address{Department of Mathematics, University of Georgia, 
Athens, GA 30602, USA}
\email{jfc@math.uga.edu}
\thanks{Research partially supported by 
Simons Foundation grant 054813-01}


\begin{document}

\begin{abstract}
Suppose that $G$ is a finite group and $k$ is a field of characteristic
$p>0$. We consider the complete cohomology ring $\CE_M^* = 
\sum_{n \in \bZ} \wExt^n_{kG}(M,M)$. We show that the ring has two 
distinguished ideals $I^* \subseteq J^* \subseteq \CE_M^*$
such that $I^*$ is bounded above in degrees, $\CE_M^*/J^*$ is bounded
below in degree and $J^*/I^*$ is eventually periodic with terms of 
bounded dimension. We prove that if $M$ is neither projective nor 
periodic, then the subring of all elements in 
negative degrees in $\CE_M^*$ is a nilpotent algebra.
\end{abstract}

\maketitle

\section{Introduction}
Let $G$ be a finite group and $k$ a field of characteristic $p >0$. 
In \cite{BC}, Dave Benson and this author defined products in the negative
cohomology $\wHH^*(G,k)$ and showed that products of elements in negative
degrees often vanish. For $G$ an elementary abelian $p$-groups, the 
product of any two elements with negative degrees is zero as well as
the product of any element in a positive degree with another in a 
negative degree. In general, for any group, the negative cohomology 
ring is a nilpotent $k$-algebra. This fact was not actually proved in 
\cite{BC}, though it follows directly from the statement about negative 
degree products in the cohomology of elementary abelian groups and
Theorem 2.5 of \cite{CQ}.
 
In this paper we extend some of the results of \cite{BC} to the complete
cohomology ring $\wExt_{kG}^*(M,M)$ of a finitely generated $kG$-module 
$M$. In particular, we prove that the $k$-algebra of elements in 
negative cohomology is nilpotent. In the course of the proof we introduce
two graded ideals $I^*$ and $J^*$, with $I^* \subseteq J^*$, whose membership 
for an element is determined by 
the polynomial rate of growth of the quotient of $\HHH^*(G,k)$ by the 
annihilator of the element. The quotient $J^*/I^*$ has bounded dimension
and is periodic in high degrees. There is another graded ideal $\FJ^*$
generated by elements in arbitrarily high negative degrees having the 
properties that $I^* \subseteq \FJ^* \subseteq J^*$ and that $\FJ^*/I^*$ 
is truly periodic. When $J^*/I^*$ is nontrivial, 
its annihilator determines a zero dimensional subvariety of 
the spectrum ${\rm Proj}(\HHH^*(G,k))$. Both it and the subvariety
corresponding to $\FJ^*/I^*$ are invariants of the module. 
The last section of the paper contains some examples of modules over
small elementary abelian groups, showing that the $J^*/I^*$ is not
trivial.

The cohomology ring $\Ext^*_{kG}(M, M)$ of a finitely generated module $M$,
is known to be a finitely generated algebra over its center and
finitely generated as a module over
$\HHH^*(G,k) \cong \Ext^*_{kG}(k,k)$ \cite{C1}. It is a PI algebra, but is
not graded commutative. Some of the constructions that we develop in this
paper,  in particular the nontriviality of the ideal $J^*$,
show up also in examples of the cohomology rings given in \cite{BC}.

The last section of the paper contains several examples for 
small elementary abelian subgroups, showing that the $J^*/I^*$ is not
trivial. 

The cohomology ring $\Ext^*_{kG}(M, M)$ of a finitely generated module $M$,
is known to be a finitely generated algebra over its center and 
as a module over 
$\HHH^*(G,k) \cong \Ext^*_{kG}(k,k)$ \cite{C1}. It is a PI algebra, but is
not graded commutative. Some of the constructions that we develop in this 
paper, in particular the nontriviality of the ideal $J^*$, 
show up also in the examples of the cohomology rings given in \cite{BC}. 

%%%%%%%%%%%%%  section 2 %%22222
\section{Notation and definitions} \label{sec:products}
Here we recall and quickly sketch the definitions of the products in 
complete cohomology. The notation introduced here is also very useful
in the examples in Section \ref{sec:examples}.

Throughout the paper, $G$ is a finite group and $k$ is a field of 
characteristic $p >0$. All $kG$-modules are assumed to be
finitely generated. The stable category $\stmod(kG)$ is the category
whose objects are finitely generated $kG$-modules. If $M$ and $N$
are $kG$-modules, then the set of morphisms from $M$ to $N$ in the stable 
category is given as 
\[
\Homul_{kG}(M,N) \quad = \quad \Hom_{kG}(M,N) / \PHom_{kG}(M,N)
\]
where $\PHom_{kG}(M,N)$ is the subset of homomorphisms from $M$ to $N$
consisting of those that factor through projective modules. 

For $M$ a $kG$-module, let $P_* = P_*(M)$ be a complete projective resolution
of $M$. That is, $P_*$ is an acyclic complex (exact sequence) of projective 
modules: 
\[
\xymatrix{
\dots \ar[r] & P_2 \ar[r]^{\partial_2} & P_1 \ar[r]^{\partial_1} & 
P_0 \ar[r]^{\partial_0} & 
P_{-1} \ar[r]^{\partial_{-1}} & P_{-2} \ar[r] & \dots
}
\]
such that the image of $\partial_0:P_0 \to P_{-1}$ is isomorphic to $M$.
Let $\Omega^n(M)$ be the image of $\partial_n:P_n \to P_{n-1}$. The 
module $\Omega^n(M)$ depends on the projective resolution. However, its
class in the stable category is a well defined object. This means that it
is well defined up to isomorphism and direct sum with a projective module
in the module category. Note that the operator $\Omega^{-1}$ is the 
translation functor or shift functor on the stable category $\stmod(kG)$. 

For any $n$, and modules $M$ and $N$, the 
cohomology $\wExt^n_{kG}(M, N)$ 
is defined to be $\wHH^n(\Hom_{kG}(P_*(M), N) \cong \Homul_{kG}(\Omega^n(M), N)$.
We note that, by translation, $\Homul_{kG}(\Omega^n(M), N) \cong 
\Homul_{kG}(\Omega^{n+s}(M), \Omega^s(N))$.  Hence, the products in the 
cohomology can be defined as the composition 
\[
\xymatrix@-.6pc{
\wExt^n_{kG}(M, N) \otimes \wExt^m_{kG}(L, M) \ar[r] &
\Homul_{kG}(\Omega^n(M), N) \otimes \Homul_{kG}(\Omega^{n+m}(L), \Omega^n(M))
}
\]
\[
\xymatrix{
{} \ar[r] & \Homul_{kG}(\Omega^{n+m}(L), N) \ar[r] & \wExt^{m+n}_{kG}(L, N)
}
\]
But notice that any homomorphism $\psi: \Omega^n(M) \to N$, 
can be lifted to chain maps $\{\psi_j\}_{j \geq n}$ and 
$\{\psi_j\}_{j < n}$: 
\[
\xymatrix@-1pc{
\dots \ar[r] & P_{n+1}(M) \ar[r] \ar[d]^{\psi_{n+1}} & 
P_{n}(M) \ar[r] \ar[d]^{\psi_n} & 
\Omega^n(M) \ar[r] \ar[d]^{\psi} & 0 
&  0 \ar[r] & \Omega^n(M) \ar[r] \ar[d]^\psi & 
P_{n-1}(M) \ar[r] \ar[d]^{\psi_{n-1}} & \dots \\
\dots \ar[r] & P_{1}(N) \ar[r] & P_{0}(N) \ar[r] & N \ar[r] & 0, 
&  0 \ar[r] & N \ar[r] & P_{-1}(N) \ar[r] & \dots 
}
\]
This works because $kG$ is a self-injective ring. 
Hence, $\psi$ defines a chain map of degree $n$ from the complete
resolution of $M$ to that of $N$, and the chain map is well defined
up to homotopy. 

Thus, we conclude that $\wExt^n_{kG}(M, N)$ is isomorphic to the 
space of homotopy classes of chain maps of degree $n$ from a complete
projective resolution of $M$ to that of $N$. And, importantly, the 
product of two cohomology elements is the homotopy class of the 
composition of their corresponding chain maps. 

%%%%%%%%%%  section 3 %%%%3333

\section{The basic ideals}
Throughout the paper, we let $k$ be a field of characteristic $p > 0$ and 
let $G$ be a finite group. For convenience we assume that $k$ is algebraically
closed. We also assume that $p$ divides the order of $G$ as otherwise 
the results of this paper are vacuous.

We fix a finitely generated $kG$-module $M$. The purpose of this section
is to introduce two ideals in $\widehat{\Ext}^*_{kG}(M,M)$ that play an 
essential role in our study. But first some notation. 

For the sake of notational economy, 
let $\hhg^n = \HHH^n(G,k)$ and $\hhg^* = \HHH^*(G,k)$. Let
$\cemp^* = \Ext_{kG}^*(M, M)$ be the cohomology ring of $M$ in nonnegative
degrees and let $\hcem^* = \widehat{\Ext}_{kG}^*(M,M) =
\sum_{n \in \bZ} \widehat{\Ext}_{kG}^n(M,M)$ be the complete cohomology 
ring. Both $\cemp^*$ and $\hcem^*$ are modules over $\hhg^*$, and
$\hcem^*$ is a module over $\cemp^*$. Moreover, $\cemp^*$ is finitely
generated over $\hhg^*$, and the homomorphism $\hhg^* \to \cemp^*$ 
given by $\zeta \mapsto \zeta\Id_M$ sends $\hhg^*$ to the center 
(in the sense of graded commutative rings) of $\cemp^*$ (for example
see \cite[Lemma 2.6]{C1}).

\begin{defi} \label{def:I}
For any $n \in \bZ$, let $I^n$ be the $k$-subspace consisting of all 
$m \in \hcem^n$ such that $\cemp^*m$ has finite dimension in 
$\hcem^*$. Let $I^* = I_M^* = \sum_{n \in \bZ} I^n$. 
\end{defi}

\begin{lemma}
Let $\tI^n$ be the subspace of $I^n$ consisting of all 
$m \in \hcem^n$ such that $\hhg^*m$ has finite dimension. Let
$\tI^* =  \sum_{n \in \bZ} \tI^n$. Then $\tI^* = I^*$.
\end{lemma}

\begin{proof}
Because the action of $\hhg^*$ on $\hcem^*$ factors through the action 
on $\cemp$, it is clear that $I^* \subseteq \tI^*$. 
Now suppose that $m \in \tI^n$
for some $n$ and $\mu \in \cemp^t$ for some $t$. Then because the
action of $\hhg^*$ commutes with that of $\cemp^*$, $\mu m \in \tI^{n+t}$. 
There exist homogeneous elements $\mu_1, \dots, \mu_s$ such that 
$\cemp^* = \sum \hhg \mu_i$. Thus $\cemp^* m = \sum \hhg^* \mu_i m$
which has finite dimension. Hence, $m \in I^*$.
\end{proof}

\begin{prop} \label{prop:Iisideal}
The subspace $I^*$ is an ideal in $\hcem^*$.
\end{prop}

\begin{proof}
It is not difficult to see that $I^*$ is closed under addition. In fact, 
$I^*$ can be characterized as the collection 
of elements $m$ in $\hcem^*$ having 
the property that  the $\cemp^*$-submodule generated by $m$ is finite
dimensional. That is, an element $m$ has this finite generation property
if and only if each of its homogeneous pieces has the property. In addition,
$I^*$ is easily seen to be closed under multiplication by elements of 
$\cemp^*$. The only thing remaining to show is that if $\mu \in \hcem^*$
has negative degree and $m$ is in $I^*$ then $\mu m$ is in $I^*$. However, 
the action of $\mu$ commutes with the action of $\hhg^*$. Hence, 
$\hhg^*\mu m= \mu \hhg^*m$ has finite dimension and is in $I^*$ by 
the last lemma. 
\end{proof}

The second ideal that we introduce is somewhat similar, but the construction
depends heavily on the fact that $\hhg^*$ is a finitely generated 
noetherian $k$-algebra. Thus we can use ideas and results from 
(graded-) commutative algebra.  
For an element $m$ in $\hcem^*$, let $\Ann(m)$ denote the annihilator of 
$m$ in $\hhg^*$. Note that if $m$ is a homogeneous element, then $\Ann(m)$ 
is a graded ideal. 

\begin{defi} \label{def:J}
For $n \in \bZ$, let $J^n$ be the $k$-subspace of $\hcem^n$ consisting of 
all elements $m$ with the property that $\hhg^*/\Ann(m)$ has Krull 
dimension at most one. This is equivalent
to the condition that there is some bound $B$, such that
for any $i$, $\hhg^i m$ has dimension at most $B$. Let 
$J^* = \sum_{n \in \bZ} J^m$.
\end{defi}

As with $I^*$, the subspace $J^*$ has a characterization in terms of the 
action of $\cemp^*$.

\begin{lemma}
Let $\tJ^n$ be the set of all elements $m$ in $\hcem^n$ such that 
there exists a number $B = B(m)$ with the property that for any 
$t>0$ the dimension of $\cemp^t m$ is at most $B$. 
Let $\tJ^* = \sum_{n \in \bZ} \tJ^n$.     Then $J^* = \tJ^*$. 
\end{lemma}

\begin{proof}
From the definition we deduce that $\tJ^* \subseteq J^*$. For the 
reverse inclusion, suppose that $m \in J^n$ for some $n$. Then there 
exists $B$ such that the dimension of $\hhg^t m$ is at most $B$ for 
any $t$. Suppose that $\mu_1, \dots, \mu_s$ is a complete collection 
of generator for $\cemp^*$ as a module over $\hhg^*$. Then 
$\cemp^* = \sum \hhg^* \mu_i$ and we see that 
$\cemp^* m = \sum_{i=1}^s \mu_i(\hhg m)$. It follows that 
the dimension of $\cemp^t m$ is at most $sB$ for any $t$. Hence,
$m$ is in $\tJ^*$. 
\end{proof}

\begin{prop}\label{prop:Jisideal}
 The subspace $J^*$ is an ideal in $\hcem^*$.
\end{prop}

\begin{proof}
The subspace $J^*$ is clearly closed under addition. Moreover, we have
seen that if $m \in J^n$ for some $n$ and $\mu \in \hcem^t$ for any $t$, 
then $\mu m$ is in $J^{m+t}$. 
\end{proof}

\begin{rem} \label{rem:periodic}
Suppose that the module $M$ is periodic. This means that $\Omega^n(M) 
\cong M \oplus P$ for some $n \neq 0$ and some projective module $P$.
It implies also that there exist $\zeta \in \hhg^n$ for 
some $n$ such that multiplication by $\zeta$
induces an isomorphism $\hcem^j \to \hcem^{j+m}$ for all $j$. Then 
provided $M$ is not projective, $J^* = \hcem^*$ and $I^* = \{0\}$. 
Conversely, if $J^* = \hcem^*$, then $M$ is periodic.
\end{rem}

\begin{rem} \label{rem:reg}
In the event that $\cemp^*$ has a regular element (as does $\HHH^*(G,k)$
\cite{Duf}), then $I^n = \{0\}$ for all $n \geq 0$.
\end{rem}

%%%%%%%%%%%%%%  section 4 %%%%%% 4444

\section{Bounds on the ideals}
In this section, we show that the submodule $I^*$ is bounded above in 
degrees while the quotient $\hcem^*/J^*$ is bounded below. The $kG$-module
$M$ is fixed and we continue the same notation as in the previous section. 
The first statement is easy to prove. 

\begin{lemma} \label{lem:bdonI}
There exists a number $B = B_I$ such that $I^n = \{0\}$ for all $n > B$.
\end{lemma}

\begin{proof}
For $n \geq 0$ let $U^n = I^n$ and $U^* = \sum_{n \geq 0} U^n = I^* \cap \cemp$. 
We know that $U^*$ is an ideal in $\cemp^*$ and is an $\hhg^*$-submodule
of $\cemp^*$. As such it is finitely generated, since $\hhg^*$ is noetherian. 
Because the submodule generated by each generator is finite dimensional,
there is an upper bound on the degrees of these submodules and a bound on the 
degrees of the whole of $U^*$.
\end{proof}

The proof for the bound on $\hcem^*/J^*$ requires a deeper analysis that
uses an idea from \cite{BC}. The bound we obtain is likely to be far from
optimal. 

\begin{thm} \label{thm:bound}
Let $\ell$ be the least common multiple of the degrees of a complete set 
of generators for the $k$-algebra $\hhg^*$. Let $d$ denote the maximum of 
the dimensions of $\Ext^j_{kG}(M,M)$ for $0 \leq j \leq \ell$.
Then for $n > d\ell$, we have that $\hcem^{-n} = J^{-n}$. 
\end{thm}

\begin{proof}
Suppose that $m \in \hcem^{-n}$ with $-n < -d\ell$ and $m$ not in $J^{-n}$. 
Then, $\hhg^*/\Ann(m)$ has Krull dimension at least two. As a consequence, 
there exist elements $\zeta_1$ and $\zeta_2$ in $\hhg^\ell$ such that 
subalgebra $A = k[\zeta_1, \zeta_2] \subseteq \hhg^*$ is a polynomial ring
with two generators and $A \cap \Ann(m) = \{0\}$. 
Hence, the collection of all 
elements $\zeta_1^i\zeta_2^jm$ are $k$-linearly 
independent for all $i,j \geq 0$. This means that, for any $u$, the subspace 
$\hhg^{u\ell}m$ has dimension at least $u+1$. 
Such a subspace is in 
$\hcem^{u\ell-n}$. 

So if we write $n = q\ell +r$ with $0 \leq r < \ell$, 
then $q \geq d$ and $\hhg^{(q+1)\ell} m \subseteq \cemp^{\ell-r}$ has dimension
at least $q+2 > d$. The contradiction proves the theorem. 
\end{proof}

For any number $b$, let $\CD_b^* = \sum_{n \geq b}
\hcem^n$. Then $\CD_b^*$ is a finitely generated module over both
$\cemp^*$ and $\hhg^*$. Let $\CJ^* = J^* \cap \CD_b^*$ and 
$\CI^* = I^* \cap \CD_b^*$. Thus, $\CJ^n = J^n$ if $n \geq b$ and $\CJ^n = 
\{0\}$ otherwise. It is similar for $\CI^n$. 
We may assume that $\CJ^*/\CI^* \neq 0$ as otherwise there is nothing
to prove. Let $\Ann_b$ be the annihilator
of $\CJ^*/\CI^*$ in $\hhg^*$. Then $\hhg^*/\Ann_b$ has Krull dimension one 
because of the finite generation.

Now choose an element $\zeta$ in $\hhg^s$ for some $s$ such that $\hhg^*/
\Ann_b$\ is a finitely generated module over $k[\zeta]$. That is, we want
that $\zeta$ is a one-element homogeneous system of parameters 
$\hhg^*/\Ann_b$. In particular, if $U$ is the ideal generated by
$\ann_b$ and $\zeta$, or by $\ann_b$ and any power of $\zeta$,  
then $\hhg^*/U$ has finite dimension. 

With this notation we can prove the following. 

\begin{prop} \label{prop:inject-z}
For any $n \geq b$, the map $J^n/I^n \to J^{n+s}/I^{n+s}$ induced by 
multiplication by $\zeta$ is injective. 
\end{prop} 

\begin{proof}
Suppose that $m \in J^n$ has the property that $\zeta m \in I^n$. 
Because $\hhg^* \zeta m$ has finite dimension,  $\zeta^t$ annihilates
$\zeta m$ for some $t$. Thus $m$ is annihilated by $\zeta^{t+1}$ as 
well as by $\Ann_b$. Hence, by the choice of $\zeta$, $\hhg^* m$ has finite
dimension and $m \in I_n$. 
\end{proof}

Because the choice of $b$ is arbitrary, the following corollary is immediate.

\begin{cor}
for any $n$, $\Dim(J^n/I^n) \leq \Dim(J^{n+s}/I^{n+s}),$ where 
$s$ is the degree of $\zeta$. 
\end{cor}

In addition we note that the choice of the number $b$ is irrelevant. 

\begin{cor}
Suppose that $b$ and $c$ are any numbers. Then $\Ann_b = \Ann_c$ is 
the annihilator of $J^*/I^*$.
\end{cor}

\begin{proof}
Suppose that $b > c$. Clearly, $\Ann_c \subseteq \Ann_b$, since 
$\CD_c \supseteq \CD_b$. So suppose that $x \in \Ann_b$ and 
$m \in \CD_c \cap J$. Then for $n$ sufficiently large $\zeta^nm
\in J \cap \CD_b$. So $\zeta^n xm = x\zeta^n m \in I^*$. Because 
multiplication by $\zeta^n$ on $J^*/I^*$ is injective, we must 
have that $xm \in I^*$. The fact that $\Ann_b$ is the annihilator 
of $J^*/I^*$ is clear from what we have proved.  
\end{proof} 

%%%%%%%%%%% section 5 %%%%%5555

\section{duality}
In this section we recall the definition of Tate duality and reveal some 
of its consequences for the structure of cohomology rings. For background on 
Tate duality see one of the books \cite{Br} (see particularly Problem 4 
on page 148), \cite{CE} (XII.6) and the discussion in \cite{BC}. As before,
the module $M$ is fixed and we continue the notation of previous sections. 

For a $kG$-module $N$, let $N^* = \Hom_k(N, k)$ be its $k$-dual.  
Tate duality, as applied to $\hcem^*$, for an integer $n$, is a 
nondegenerate pairing 
\[
\wExt_{kG}^n(M,M) \otimes \wExt_{kG}^{-n-1}(M,M) \to k
\]
Using the adjointness and the standard isomorphism 
\[
\wExt_{kG}^n(M,M) \cong \wExt_{kG}^n(k,M^* \otimes M) \cong 
\wHH^n(G, \Hom_k(M, M))
\]
the duality is induced by the composition and is the usual cup 
product.  
\[
\xymatrix{
\wHH^n(G, \Hom_k(M, M)) \otimes \wHH^{-n-1}(G, \Hom_k(M, M))
\ar[r] & {} 
}
\]
\[
\xymatrix{
\wHH^{-1}(G, \Hom_k(M, M)) \ar[r] &  \wHH^{-1}(G, k) \cong  k
}
\]
where the first map is the product  and the second 
is induced by the trace map 
$\Hom_k(M,M) \to k$. The upshot of this is that if $\alpha$, $\beta$ and 
$\gamma$ are three homogeneous elements of $\wExt_{kG}^*(M,M)$ whose
degrees sum to $-1$, then 
\[
\langle \alpha\beta, \gamma \rangle \ = \ \langle \alpha, \beta\gamma 
\rangle.
\]
All of this leads us to the following.

\begin{prop} \label{prop:plusminus}
Suppose that $\alpha$ is a homogeneous element of $\hcem^*$ having 
negative degree. Then there exists an element $\beta\in \hcem$ 
with negative degree such that $\alpha\beta \neq 0$, if and only if
there exists an element $\gamma$ such that $\gamma\alpha \neq 0$
has positive degree. In either case we may assume that 
$\deg(\alpha) + \deg(\beta) + \deg(\gamma) = -1$.
\end{prop}

\begin{proof}
Suppose that there exists $\beta$ with $\deg(\beta) < 0$ such 
that $\alpha\beta \neq 0$. Then  there is an element $\gamma$ 
such that $\deg(\gamma) = -\deg(\alpha) - \deg(\beta) -1$ and 
$\langle \gamma, \alpha\beta \rangle \neq 0$.
Thus $\langle \gamma\alpha, \beta \rangle \neq 0$, and 
$\gamma\alpha \neq 0$. The reverse statement is proved by the
reverse argument. The statement about the degrees is obvious from
the construction.  
\end{proof}

\begin{cor} \label{cor:nilI1}
Let $B_I$ be the bound such that $I^n = 0$ if $n > B_I$ from Lemma
\ref{lem:bdonI}. Suppose that $\alpha \in I^n$ with $n < -B_I-1$. 
Then $\beta\alpha = 0$ for every element $\beta \in I^*$. 
\end{cor}

\begin{proof}
Choose any homogeneous $\beta \in I$. Suppose that 
$\beta\alpha \neq 0$. Let $\gamma$ be as in the proposition. That is, 
$\gamma\beta \neq 0$ and $\deg(\gamma) = -\deg(\alpha) -\deg(\beta) -1.$
Then $\deg(\gamma\beta) = -\deg(\alpha) -1 > B_I$ implying that 
$\gamma\beta = 0$. The contradiction proves the corollary. 
\end{proof}

\begin{thm} \label{thm:nilpotentI}
The ideal $I^*$ is a nilpotent ideal in $\hcem^*$. 
\end{thm}

\begin{proof}
Let $\gamma_1, \gamma_2, \dots$ be a sequence of homogeneous elements
in $I^*$. Our object is to show that for any $n$, sufficiently large,
the product $\gamma_1 \gamma_2 \cdots \gamma_n = 0$. For $i \geq 2$, 
let $a_i$ be 
the degree of $\mu_i = \gamma_2 \cdots \gamma_i$. Note that if for 
any $i$, $a_i > B_I$, then $\mu_i = 0$ by Lemma \ref{lem:bdonI}.
Likewise, if $a_i < -B_I-1$ then $\gamma_1\mu_{i} = 0$ by Corollary 
\ref{cor:nilI1}. Consequently, if $\gamma_1\mu_i \neq 0$ for all $i$ we 
must have that $B_I \leq a_i \leq -B_I-1$. 

The ring $\widehat{\cemp}^0 = \wExt^0_{kG}(M,M)$ is local, because $M$ is 
indecomposable. Let $N$ be the nilpotence degree of its radical. 
That is, any product of at least $N$ elements in the radical of 
$\widehat{\cemp}^0$ vanishes. We claim the the nilpotence degree of 
$I$ is at most $2(N+1)(B_I+1)$. For suppose that $n > 2(N+1)(B_I+1)$. Then
by the pigeonhole principle, at least $N+1$ 
of the $a_i$'s for $2 \leq i \leq n$
must be the same. That is, there exist $1 < i_0 < i_1 < \dots < i_N$ 
such that $a_{i_0} = a_{i_1} = \dots = a_N$. This means that 
the elements $\vartheta_j = \gamma_{i_j+1}\cdots \gamma_{i_{j+1}}$
all have degree $0$ and hence are elements in the radical of $\hcem^0$.
Consequently, $\vartheta_1 \cdots \vartheta_N = 0$ and also 
$\mu_n =0$. This proves the theorem. 
\end{proof}

\begin{thm} \label{thm:nilJ}
Suppose that $M$ is an indecomposable, nonprojective $kG$-module that
is not periodic. 
Let $\CK^* = \sum_{n<0} J^n$ be the negative cohomology that is in the ideal
$J^*$. Then $\CK^*$ is a nilpotent algebra.
\end{thm}

\begin{proof}
Let $N$ denote the nilpotence degree of the radical of $\cemp^0 
= \Ext^0_{kG}(M,M).$ Let $\zeta\in \hhg^*$ be an element as in 
Proposition \ref{prop:inject-z}. Let $n = \deg(\zeta).$ Suppose that 
$\gamma_1, \gamma_2, \dots$ is a sequence of homogeneous elements
in $\CK^*$. Let $a_i = \deg(\gamma_1\cdots \gamma_i)$ and write 
$a_i = q_in +r_i$ where $0 \leq r_i < n$. Assume that 
$m \geq n(N+1)$.

By the pigeonhole principle there exist $i_1 < i_2 < \dots < i_{N+1}$ 
such that $r_{i_1} = \dots = r_{i_{N+1}}$. Let $\mu_j = \gamma_{i_j +1} \cdots 
\gamma_{i_{j+1}}$. Then $\deg(\mu_j) = s_jn$ for some $s_j < 0$. 
Also, $\vartheta_j = \zeta^{-s_j}\mu_j$ has degree $0$ 
and is an element of
$\hcem^0$. Because $M$ in neither projective nor periodic, $\vartheta_j$
factors through some $\Omega^t(M) \not\cong M$ for some $t$, and 
$\vartheta_j$ is in the radical of $\hcem^0$ for every $j$.
Let $s = \sum s_i$. Then 
\[
\zeta^s\gamma_1 \cdots \gamma_{i_{N+1}} = 
\gamma_1 \dots \gamma_{i_1} \vartheta_1 \dots \vartheta_N  \in I^*
\]
by Proposition \ref{prop:inject-z}. By Theorem \ref{thm:nilpotentI}
for $m$ sufficiently large, $\gamma_1 \dots \gamma_m = 0$.
\end{proof}

\begin{thm}  \label{thm:negcoho}
Suppose that $M$ is a $kG$-module that is neither projective nor periodic.
Let $\cemm^* = \sum_{n<0} \wExt^n_{kG}(M,M)$ be the algebra of negative 
cohomology of $M$. Then $\cemm^*$ is a nilpotent algebra. 
\end{thm}

\begin{proof}
By Theorem \ref{thm:bound}, the product of any sufficiently large number
of elements on $\cemm^*$ lies in $J^*$ and in negative degrees. Hence, the 
proof follows from Theorem \ref{thm:nilJ}.
\end{proof}

%%%%%%%%% section 6  %%%%66666

\section{The periodic streak} \label{sec:streak}
In the case that $M = k$, there are examples of groups where multiplication 
by the element $\zeta$ as in Proposition \ref{prop:inject-z} is also 
surjective on $J^*/I^*$. These include the case that
$G$ is a semidihedral $2$-group (see Section 4 of \cite{BC}). However, we 
see no reason for $\zeta$ to be an isomorphism on $J^*/I^*$ in 
general. 

In this section, we show that with a modification $\FJ^*$ 
of the ideal $J^*$, there is a periodic streak 
$\FJ^*/I^*$ that runs through the entire cohomology ring. 
Moreover, if $M$ is not a periodic module, then 
the ideal $\FJ^*$ is nilpotent also in positive degrees. 
In the next section we show that this streak may be nonzero even for 
modules over elementary abelian $p$-groups. First we give the definition. 
For the fixed $kG$-module $M$, let $I^*$ and $J^*$ be as in Definitions 
\ref{def:I} and \ref{def:J} for the complete cohomology ring $\hcem^*$.

\begin{defi} \label{def:FJ}
For any integers $t$ and $n$, let 
\[
\FJ_t^n  = I^n + \hcem^n \cap (\hcem^* \sum_{m \leq t} \hcem^m).
\]
That is, $\FJ_t^*$ is the $\hcem^*$-submodule (left ideal) of $\hcem^*$
that is the sum of $I^*$ and the submodule generated by elements in 
degrees at most $t$. Let $\FJ^n = \cap_{t<0} \FJ^n_t$. 
\end{defi}

Then we have the following. 

\begin{lemma} \label{lem:basicFJ}
For any $n$, $\FJ^n \subseteq J^n$. In addition, $\FJ^*$ is a
left ideal. 
\end{lemma}

\begin{proof}
Choose $t$ sufficiently large negatively that $J^s = \hcem^s$
for all $s \leq t$. Such a $t$ exists by Theorem \ref{thm:bound}. 
Then $\FJ^n \subseteq \FJ_t^n \subseteq J^n$. This proves the 
first statement. For the second, note that $\FJ_t^*$ is a left ideal
for any $t$. So $\FJ^*$, which is the intersection, is also a left ideal. 
\end{proof}

\begin{thm}
Assume that $M$ is neither periodic nor projective. 
Let $\zeta \in \hhg^s$ be as in Proposition \ref{prop:inject-z}. 
Then, multiplication by $\zeta$ is an isomorphism on $\FJ^*/I^*$.  
In addition, $\FJ^*$ is a nilpotent two-sided ideal.  
\end{thm}

\begin{proof} 
From the lemma we see that multiplication by $\zeta$ is injective on
$\FJ^*/I^* \subseteq J^*/I^*$. For any $n$, the dimension of 
$J^n/I^n$ is smaller than that of $J^{n+s}/I^{n+s}$, where $s$ is 
the degree of $\zeta$. Hence for some $t$ sufficiently large, negatively,
it must be that $\Dim(J^n/I^n) = \Dim(J^{n-s}/I^{n-s})$ for all 
$n< t$. Hence, if $n < t$, then multiplication by $\zeta$ is an 
isomorphism $J^{n-s}/I^{n-s} \to J^n/I^n$. The implication is that
for $n < t$, $\FJ^n = J^n$, and moreover, $\FJ^* = \FJ_t^*$. 

So suppose that $\gamma$ is in $\FJ^n$ for some $n$. Then by the 
definition, $\gamma = \mu + \nu$ where $\mu$ is in $I^n$ and 
$\nu$ is a finite sum of elements of the form $\sigma\tau$ for 
$\sigma \in \hcem^{n-r}$ and $\tau \in \FJ_t^r$ for some $r < t$. 
However, by the above argument, $\tau = \mu^\prime + \zeta\tau^\prime$
for $\mu^\prime \in I^r$ and $\tau^\prime$ in $\FJ^{r-s}$. Consequently, 
the class of $\gamma$ modulo $I^*$ is the same as the class of 
$\sigma\zeta\tau^\prime$ which is the same as that of 
$\zeta(\sigma\tau^\prime)$. So the class of $\gamma$ in   
$\FJ^*/I^*$ is a multiple of $\zeta$ and multiplication by
$\zeta$ is surjective on $\FJ^*/I^*$.

Suppose that $\gamma \in \FJ^n$ and $\sigma \in \hcem^t$ for some $n$ 
and $t$. Then for any $m >0$, there exists $\gamma = \mu +\zeta^m \nu$
with $\mu \in I^n$ and $\nu \in \FJ^{n-sm}$ by what we have just proved. 
Then $\gamma\sigma = \mu\sigma + \zeta^m(\nu\sigma)$, where 
$\nu\sigma \in J^{n-sm+t}$. Assuming that $m$ is sufficiently large, 
$J^{n-sm+t} = \FJ^{n-sm+t}$ and $\gamma\sigma \in \FJ^*$. Hence, 
$\FJ^*$ is a two-sided ideal. 

The nilpotence of $\FJ^*$ is evident. For suppose that 
$x_1, \dots, x_m$ are elements of $\FJ^*$. Write each as 
$x_i = \mu_i + \zeta^{n_i}y_i$ for $\mu_i \in I^*$ and $n_i$ large 
enough that $y_i \in \FJ^*$ has negative degree. Then the product
$x_1\cdots x_m$ equal to $\zeta^a (y_i \cdots y_m)$ modulo $I^*$,
where $a = \sum n_i$. So, by Theorem \ref{thm:negcoho},  
if $m$ is large enough the product will
lie in $I^*$ which we know to be nilpotent. 
\end{proof}

%%%%%%%%%%%% section 6 %%%%%66666

\section{Examples} \label{sec:examples}
The aim of this section is to  show by examples 
that the ideal $\FJ^*$ is not trivial, {\it i. e.}
not equal to $I^*$, even in some very simple cases. 
Two examples are presented, one in characteristic 2 
and another in odd characteristics. The second is presented with
somewhat less detail. 

\subsection{A characteristic 2 example.} \label{ex:even}
Suppose that $G = \langle x, y, z \rangle$ is an elementary 
abelian group of order $8$, and let $k$ be a field of characteristic $2$.
Then $kG \cong k[X,Y,Z]/(X^2, Y^2, Z^2)$ 
where $X = x-1$, $Y = y-1$ and $Z = z-1$. A $k\langle x \rangle$-projective
resolution of $k$ has the form 
\[
\xymatrix{
\dots \ar[r] & P_2^X \ar[r]^{X} & P_1^X \ar[r]^{X} & 
P_0^X \ar[r]^{\varepsilon} & k \ar[r] & 0
}
\]
where for every $j \geq 0$, $P_j^X \cong k\langle x \rangle \cong k[X]/(X^2)$,
and the boundary maps $P_{j+1}^X \to P_j^X$ are multiplication
by $X$. Substituting $Y$ for $X$ and $y$ for $x$, we get a projective 
$k\langle y \rangle$-resolution $P_*^Y$ of $k$. Then a projective 
$kG$-resolution of $k$ is the tensor product 
$P_* = P_*^X \otimes P_*^Y \otimes P_*^Z$ since $kG \cong k\langle x \rangle
\otimes k\langle y \rangle \otimes k\langle z \rangle$. 

To obtain a complete resolution we take the $k$-dual of $P_*$, shift by 
one degree and splice. The dual of $P_*^X$ has the form 
\[
\xymatrix{
0 \ar[r] & k \ar[r]^{\varepsilon^\sharp} & P_{-1}^X \ar[r]^{X} & 
P_{-2}^X \ar[r]^{X} & P_{-3}^X \ar[r] & \dots
}
\]
Here, since $k\langle x \rangle$ is self-dual, we have that 
$P^X_{-j-1} = (P^X_j)^* $ for all $j \geq 0$ and the boundary maps are the 
duals of the maps in the projective resolution. 
Thus, a complete projective $kG$-resolution of $k$ has the form
\[
\xymatrix{
\dots \ar[r] & P_2 \ar[r] & P_1 \ar[r] & 
P_0 \ar[r]^{XYZ} & P_{-1} \ar[r] & P_{-2} \ar[r] & \dots
}
\]
where the boundary map $P_0 \to P_{-1}$ is the splice
$\varepsilon^\sharp\varepsilon:P_0^X \otimes P_0^Y \otimes P_0^Z \to 
P_{-1}^X \otimes P_{-1}^Y \otimes P_{-1}^Z$, taking $1 \otimes 1 \otimes 1$
to $X\otimes Y \otimes Z$. 

We define the module $M$ as follows. Let $F = kGu \oplus kGv$ be the free
$kG$-module with free basis $\{u, v\}$. Then $M = F/L$ where 
$L = \langle Xu, ZYu, Yu-Xv \rangle$. Then $M$ has a basis consisting of 
the classes modulo $L$ of the elements of the set $B = \{Zu, u, Yu, v, Zv, 
Yv, ZYv\}$. A diagram for the module looks like 
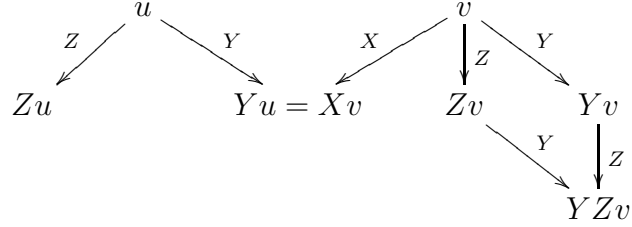
\[
\xymatrix{
& u \ar[dl]_Z \ar[dr]^Y && v \ar[dl]_X \ar[d]^Z \ar[dr]^Y \\
Zu && Yu=Xv & Zv \ar[dr]^Y & Yv \ar[d]^Z \\
&&&& YZv 
}
\]
where the vertices are a $k$-basis and the arrows indicate multiplications
by the designated elements. 

In what follows, we use the isomorphism $\wExt^*_{kG}(M,M) \cong 
\wHH^*(G, \Hom_k(M,M))$. In addition, for convenience of notation, 
we view $kG$ as a truncated polynomial ring in 
variables $X, Y, Z$. This has the Hopf algebra structure we would 
get by regarding $kG$ as the restricted enveloping algebra of a three
dimensional commutative restricted Lie algebra. In particular, 
for $f \in \Hom_k(M,M)$, $m \in M$ and $U$ any $k$-linear combination
of $X, Y, Z$, we have that $(Uf)(m) = Uf(m)-f(Um)$. We see from 
the results of \cite{CI} that changing the coalgebra structure
makes essentially no difference in the example that we compute.
That is, the action of the subring of $H^*(G,k)$ generated by the 
Bocksteins of the degree-one elements on $\Ext^*_{kG}(M,M)$
is the same regardless of which Hopf algebra structure is chosen. 

For notation, let $u_{a,b,c}$ denote the element $1 \otimes 1\otimes 1$
in $P^X_a \otimes P^Y_b \otimes P_c^Z$. Then for $n\geq 0$, $P_n$ is
generated by all $u_{a,b,c}$ with $a+b+c = n$ and $a, b, c$ all nonnegative. 
for $n < 0$, $P_n$ is generated by all $u_{a,b,c}$ with $a+b+c = n-2$ 
and $a, b, c$ all negative.

Define $f, g \in \Hom_k(M,M) = \End_k(M)$ 
by $f(Zu) = v$ and $g(Zu) = u$ and 
$f(w) = 0 = g(w)$ for $w$ any element in the basis $B$ other than $Zu$.
For $n < -1$, we define a map $\alpha_n: P_{n} \to \End_k(M)$ by 
\[
\alpha_{n}(u_{a, b, c}) = \begin{cases} 
f & \text{ for } (a,b,c) = (n,-1,-1), \\
g & \text{ for } (a,b,c) = (n+1,-2,-1), \\
0 & \text{otherwise}.
\end{cases}
\]
Note that, with $n < -1$,  $\partial(u_{n+1,-1,-1}) = 
Xu_{n,-1,-1} + Yu_{n+1,-2,-1} + Zu_{n+1,-1,-2}$.
So that $\alpha_{n}(\partial(u_{n+1,-1,-1})) = Xf +Yg = 0.$ Thus, with
some additional effort, it can 
be seen that $\alpha_{n}$ is a cocycle. 

For $n \geq -1$, define $\alpha_n: P_n \to \End_k(M)$ as follows:
\[
\alpha_{n}(u_{a, b, c}) = \begin{cases} 
f  & \text{ for } n = -1 \text{ and } (a,b,c) = (-1,-1,-1), \\
YZf & \text{ for } n \geq 0 \text{ and } (a,b,c) = (n,0,0), \\
0 & \text{otherwise}
\end{cases}
\]
The function $YZf$ has values $(YZf)(u) = Yv$, $(YZf)(Zu) = YZv$ and 
$(YZf)(w) = 0$ for $w$ any other element of the basis $B$. Notice 
that $YZf$ is a $kG$-homomorphism. 
The fact that every $\alpha_n$ is a cocycle is a consequence of the next
result.

\begin{prop} \label{prop:ex1}
 Let $\zeta \in \hhg^1$ be the element represented by the 
cocycle $\zeta: P_1(k) \to k$ by $\zeta(u_{1,0,0}) = 1$ and 
$\zeta(u_{0,1,0}) = 0 = \zeta(u_{0,0,1})$. Let $\tilde{\alpha}_n$
denote the cohomology class in $\hcem_n$ of $\alpha_n$.
Then $\tilde{\alpha}_n\zeta = \tilde{\alpha}_{n+1} \neq 0$ for all $n$.
In particular, $\tilde{\alpha}_n \in \FJ^n$ and is not in $I^n$. 
\end{prop}

\begin{proof}
The trick is to write out the chain map corresponding to $\zeta$. For 
specific $n$ assume that if $n \geq 0$, then $a+b+c =n$  and that 
$a, b, c$ are all nonnegative while if $n < 0$ then 
$a+b+c = n-2$ and $a, b, c$ are all negative. 
Then the chain map for $\zeta$ is 
given by 
\[
\zeta_n(u_{a,b,c}) = \begin{cases}
u_{a-1, b, c} & \text{ if } n > 0 \text{ and } a \neq 0, \\
0 & \text{ if } n > 0 \text{ and } a = 0, \\
YZu_{-1,-1,-1} & \text{ if } n = 0, \\
u_{a-1, b, c} & \text{ if } n < 0. \\
\end{cases}
\]
Note that in degree zero we have that 
\[
\partial_{-1}\zeta_0(u_{0,0,0}) 
= \partial_{-1}(YZu_{-1,-1,-1})) = XYZu_{-2,-1,-1}=
\]
\[ 
XYZ\zeta_{-1}(u_{-1,-1,-1}) =
\zeta_{-1}\partial_{0}(u_{0,0,0}). 
\]
The remaining checks that this 
is a chain map are even easier. 
Verifying that $\tilde{\alpha}_n\zeta = \tilde{\alpha}_{n+1}$ is 
straightforward. 

For $n >0$ notice that $\partial_n(u_{n,0,0}) =
Xu_{n-1,0,0}$. So if $\alpha_n$ is a coboundary, then $YZf =
\alpha_n(u_{n,0,0}) = \beta\partial_n(u_{n,0,0}) = X\beta(u_{n-1,0,0})$
for some $\beta$. However, we can check that this does not happen, since
$YZf$ is not in $X\End_k(M)$. Thus $\tilde{\alpha}_n \neq 0$. 
\end{proof}

\subsection{An odd characteristic example} \label{ex:odd}
In this example we assume that $k$ has prime characteristic $p>2$ and that
$G = \langle x, y \rangle$ is a elementary abelian group of order $p^2$.
So $kG \cong k[X, Y]/(X^p, Y^p)$ is a truncated polynomial ring, 
where $X = x-1$ and $Y = y-1$. Note this time that the projective 
resolution of $k$ as a $k\langle x \rangle$-module has the form 
\[
\xymatrix{
\dots \ar[r] & P_2^{X} \ar[r]^{X^{p-1}} & P_1^X \ar[r]^{X} &
P_0^X \ar[r]^{\varepsilon} & k \ar[r] & 0.
}
\]
That is, the even dimensional boundary maps are multiplication by 
$X^{p-1}$ rather than $X$. 

We use a variant of the notation of the last example. In 
particular, $u_{i,j}$ denotes the element $1 \otimes 1$ generating 
$P_i^X \otimes P^Y_j$. If $n \geq 0$, then $P_n$ is a direct sum of 
all $P_i^X \otimes P^Y_j$ with $i$ and $j$ nonnegative and 
$i+j = n$. If $n < 0$, then $P_n$ is a direct sum of 
all $P_i^X \otimes P^Y_j$ with $i$ and $j$ both negative and 
$i+j = n-1$. 

The module is $M \cong (kGu \oplus kGv)/L$ where $L$ is generated by
$X^2u$, $Yu-Xv$, $XYu$, $Y^2u$. Hence, $M$ has a basis consisting of the 
classes modulo $L$ of the elements in the set $B = \{Xu, u, Yu = Xv, 
v, Yv, \dots, Y^{p-1}v\}$. Let $f: M \to M$ be given by the rule 
that $f(Xu) = v$ and $f(w) = 0$  for $w$ any element of $B$ other 
than $Xu$. Notice that $X^3f = 0$.

Assume first that $p \geq 5$. Define $\alpha_n: P_{2n-1} \to \End_k(M)$ by 
\[
\alpha_n(u_{a,b}) = \begin{cases} 
f & \text{ if } (a,b) = (2n-1,-1) \text{ and } n \leq 0 \\
Y^{p-1} f & \text{ if } (a,b) = (2n-1, 0) \text{ and } n > 0 \\
0 & \text{ otherwise }
\end{cases}
\]

In the case that $p = 3$ some adjustment must be made because of the 
fact that $X^2f \neq 0$. Let $g, h \in \End_k(M)$ be the functions
given by $g(Yu) = Yu$, $h(Y^2v) = Yu$, and $g(w) = 0 = h(t)$ for 
$w,t \in B$, $w \neq Yu$, $t \neq Y^2u$. 
Then, in case $p =3$, let $\alpha_n$ be given by 
\[
\alpha_n(u_{a,b}) = \begin{cases}
f & \text{ if } (a,b) = (2n-1,-1) \text{ and } n \leq 0,  \\
g & \text{ if } (a,b) = (2n,-2) \text{ and } n \leq -1,\\
h & \text{ if } (a,b) = (2n+1,-3) \text{ and } n \leq -1, \\
Y^{p-1} f & \text{ if } (a,b) = (2n-1, 0) \text{ and } n > 0, \\
0 & \text{ otherwise }
\end{cases}
\]
We emphasize that $\alpha_n$ is a cocycle in degree $2n-1$. The result
is the following.

\begin{prop} \label{prop:ex2}
 Let $\zeta \in \hhg^2$ be the element represented by the
cocycle $\zeta: P_2(k) \to k$ by $\zeta(u_{2,0}) = 1$ and
$\zeta(u_{1, 1}) = 0 = \zeta(u_{0,2}).$ 
Let $\tilde{\alpha}_n$
denote the cohomology class in $\hcem_{2n-1}$ of $\alpha_n$.
Then $\tilde{\alpha}_n\zeta = \tilde{\alpha}_{n+1} \neq 0$ for all $n$.
In particular, $\tilde{\alpha}_n \in \FJ^n$ and $\tilde{\alpha_n}$ 
is not in $I^n$.
\end{prop}

It can be computed that a chain map 
of $\zeta$ on the complete
resolution of $k$ is given by the following. Here $\zeta_n: P_n \to P_{n-2}$.
For any $n$ and $a,b$, let 
\[
\zeta(u_{a,b}) = \begin{cases}
u_{a-2, b} & \text{ if } n > 1 \text{ and } a >1, \\
0 & \text{ if } n > 0, a = 0, 1 \text{ and } b > 0,  \\
Y^{p-1}u_{-1,-1} & \text{ if } (a,b) = (1,0),\\
Y^{p-1}u_{-2,-1} & \text{ if } n =0 \text{ and } (a,b) = (0,0), \\
u_{a-2, b} & \text{ if } n < 0. \\
\end{cases}
\]
Again, we leave it to the reader to prove that $\{\zeta_n\}$ is a 
chain map and that the proposition holds.

\subsection{One more example} \label{subsec:ex3}
Finally, we present an example of a $kG$-module $M$ and an element of 
$\gamma$ of $\wExt^*_{kG}(M,M)$ in degree $-1$ where $\gamma\hgs$
is more than periodic meaning that $J^{-1} \neq \hcem^{-1}$. 
The group and notation for the cohomology are
the same as in the first example \ref{ex:even}. In particular, $p = 2$
and the projective resolution of $k$ is the tensor product 
$P_* = P_*^X \otimes P_*^Y \otimes P_*^Z$. 

Let $M$ be the quotient module $M = F/L$, where $F = kGu \oplus kGv$ is
the free $kG$-module on generators $u$ and $v$, and $L$ is generated 
by the elements $Zu$, $Yu - Xv$, and $YZv$. It has basis consisting of 
the classes modulo $L$ of the elements $B = \{ Xu, u, Yu, v, Yv, Zv \}$.
Define $f \in \End_k(M)$ by setting $f(Xu) = v$ and $f(w) = 0$ for 
$w$ any other element of $B$. Let $\gamma \in \hcem^{-1}$ be the element
represented by the cocycle $\gamma(u_{-1,-1,-1}) = f$.

Let $\zeta_Y \in \HHH^1(G,k)$ be represented by the cocyle 
$\zeta_Y: P_1 \to k$,  
such that $\zeta_Y(u_{0,1,0}) = 1$, $\zeta_Y(u_{1,0,0}) = 0 = 
\zeta_Y(u_{0,0,1})$. Let $\zeta_Z$ be given by $\zeta_Z(u_{0,0,1}) = 1$, 
$\zeta_Z(u_{1,0,0}) = 0 = \zeta_Z(u_{0,1.0})$.
Notice that $\gamma\zeta_Y(u_{0,0,0}) = XZf = g_1$ and 
$\gamma\zeta_Z(u_{0,0,0})  = XYf = g_2$, where $g_1(u) = Zv$, 
$g_2(u) = Yv$ and $g_1(w) = 0 = g_2(w)$ for $w \in B$, $w \neq u$. 
Note that $g_1$ and $g_2$ are $kG$-homomorphisms. 

We can prove the following.

\begin{prop} \label{prop:ex3}
With the given module $M$ and the above notation, we have that 
homomorphism $\hgs \to \hcem^*$ given by multiplication by $\gamma$ 
is injective on the polynomial subring $k[\zeta_Y, \zeta_Z]$. 
In particular, $\gamma \in \hcem^{-1}$, and $\gamma \notin J^{-1}$.
\end{prop}

\begin{proof}
We leave most of the details to the reader. 
As in the other examples, one proceeds by finding representing 
chain maps on a complete resolution for the cohomology 
elements $\zeta_Y$ and $\zeta_Z$. This is very similar to the
chain map for the element $\zeta$ in the proof of Proposition 
\ref{prop:ex1}. One important item to notice is 
that if $\alpha: P_0 \to \End_k(M)$ is given by $\alpha(u_{0,0,0})
= Xf$, then $\alpha\partial(u_{0,1,0}) = g_1$ and 
$\alpha\partial(u_{0,0,1}) = g_2$. This means that the 
cocycles $\beta_1$ and $\beta_2$ given by $\beta_1(u_{0,1,0}) = g_1$,
$\beta_2(u_{0,0,1}) = g_2$, $\beta_1(u_{1,0,0}) = \beta_1(u_{0,0,1})
= \beta_2(u_{1,0,0}) = \beta_1(u_{0,1,0}) = 0$, differ by a 
coboundary. These cocycles both represent $\gamma\zeta_Y\zeta_Z$. 
\end{proof}


\begin{thebibliography}{10}

\bibitem{BC}
D. J. Benson and J. F. Carlson, {\em Products in negative cohomology},
J. Pure Appl. Algebra {\bf 82}(1992), 107\--129.

\bibitem{Br}
K. Brown, {\em The Cohomology of Groups}, Graduate Texts in Mathematics
{\bf 87} Springer-Verlag, New York, 1982.

\bibitem{C1} J. F. Carlson, {\it The cohomology ring of a module}, J. Pure
Appl. Algebra, {\bf 36}(1985), 105\--121.

\bibitem{CQ} J. F. Carlson, {\it Cohomology and induction from elementary 
abelian subgroups}, Quarterly J. Math. {\bf 51}(2000), 169\--181.

\bibitem{CI} J. F. Carlson and S. Iyengar, {\em Hopf algebra structures and
tensor products for group algebras}, New York J. Math. {\bf 23}(2017), 1--14.

\bibitem{CE}
H. Cartan and S. Eilenberg, {\em Homological Algebra}, Princeton University
Press, London, 1956.

\bibitem{Duf}
J. Duflot, {\em Depth and equivariant cohomology},
Comm. Math. Helvetici {\bf 56}(1981), 627-637.

\end{thebibliography}
\end{document}